\documentclass[12pt]{amsart}
\usepackage{amssymb}
\usepackage{epsfig}
\usepackage{graphicx}

\theoremstyle{plain}
\newtheorem{prop}{Proposition}
\newtheorem{theo}[prop]{Theorem}
\newtheorem{coro}[prop]{Corollary}

\newtheorem{lemm}[prop]{Lemma}

\theoremstyle{definition}

\newtheorem*{conj}{Conjecture}

\newtheorem{rema}[prop]{Remark}

\newcommand{\Cone}{\mathrm{Cone }}

\newcommand{\Hom}{\mathrm{Hom}}

\newcommand{\Pic}{\mathrm{Pic}}
\newcommand{\Bl}{\mathrm{Bl}}

\newcommand{\Gr}{\mathrm{Gr}}
\newcommand{\End}{\mathrm{End}}

\newcommand{\GL}{\mathrm{GL}}

\newcommand{\ra}{\rightarrow}
\newcommand{\lra}{\longrightarrow}

\newcommand{\bP}{{\mathbb P}}

\newcommand{\bC}{{\mathbb C}}

\newcommand{\bG}{{\mathbb G}}

\newcommand{\bQ}{{\mathbb Q}}
\newcommand{\bR}{{\mathbb R}}

\newcommand{\bZ}{{\mathbb Z}}
\newcommand{\cC}{{\mathcal C}}

\newcommand{\cM}{{\mathcal M}}
\newcommand{\cN}{{\mathcal N}}
\newcommand{\cO}{{\mathcal O}}
\newcommand{\cP}{{\mathcal P}}

\newcommand{\ocM}{\overline{\mathcal M}}
\newcommand{\cbP}{\check{\mathbb P}}

\newcommand{\fl}{{\mathfrak{l}}}

\newcommand{\fS}{{\mathfrak{S}}}

\author{Brendan Hassett and Yuri Tschinkel}
\title[Flops and cubics]{Flops on holomorphic symplectic fourfolds
and determinantal cubic hypersurfaces}

\begin{document}
\date{\today}

\maketitle

\section{Introduction}
Let $\Sigma$ be a K3 surface.  
Any birational map $\Sigma \dashrightarrow \Sigma$
extends to an automorphism;  this follows from
the uniqueness of minimal models for surfaces
of non-negative Kodaira dimension.  
By the Torelli Theorem,
the group of automorphisms of $\Sigma$ isomorphic to the group of automorphisms of $H^2(\Sigma,\bZ)$
compatible with the intersection pairing $\left<,\right>$ and 
the Hodge structure on $H^2(\Sigma,\bC)$, and preserving the cone of 
nef (numerically eventually free)
divisors. The nef cone admits an intrinsic combinatorial description 
(see, for example, \cite{LP}), once we specify 
a polarization $g$:  A divisor
$h$ on $\Sigma$ is nef if and only if $\left<h,D\right> \ge 0$ for
each divisor class $D$ with $\left<D,D\right> \ge -2$ and $\left<g,D\right>>0$.  
This characterization of the automorphism group has many interesting applications 
to arithmetic and geometric questions.

In this paper, we study certain aspects of the
birational geometry of higher-dimensional analogs
of K3 surfaces, i.e., irreducible holomorphic symplectic varieties $F$.
These share many geometric properties with K3 surfaces.  For example,
the group $H^2(F,\bZ)$ carries a canonical integral quadratic form 
$\left(,\right)$, the {\em Beauville-Bogomolov} form (see, for
example, \cite{Hu99}).  Its definition uses the symplectic form
on $F$ but it can be characterized by the fact that the 
self-intersection form on $H^2(F,\bZ)$ is proportional to a power
of the Beauville-Bogomolov form \cite[1.11]{Hu99} \cite{Fujiki}
$$D^{\mathrm{dim}(F)}=c_F \left(F,F\right)^{\mathrm{dim}(F)/2}.$$
Moreover, these varieties satisfy local Torelli theorems \cite{Bea85}
and surjectivity of the period map \cite{Hu99}.  
In contrast to the surface case, 
$F$ may have numerous minimal
models and may admit birational self-maps which are not regular.
Furthermore, we lack a full Torelli theorem that would allow us
to read off automorphisms and birational self-maps from the cohomology.  

Perhaps the best-known examples of irreducible holomorphic symplectic
varieties are punctual Hilbert schemes of K3 surfaces and their
deformations \cite{Bea85}.  Here we focus on the case of length-two subschemes,
which are isomorphic to the symmetric square of the K3 surface blown-up
along the diagonal.  These also arise as varieties of lines on cubic
fourfolds \cite{BeauDonagi}.  
By \cite{HT07}, given a polarization $g$ on $F$, a divisor
$h$ on $F$ is nef if $\left(h,\rho\right) \ge 0$ for each divisor class $\rho$
satisfying  
\begin{enumerate}
\item{$\left(g,\rho\right)>0$; and}
\item{$\left(\rho,\rho\right) \ge -2$, or
$\left(\rho,\rho\right)=-10$ and
$\left(\rho,H^2(F,\bZ)\right)=2\bZ$. (These are called {\em $(-10)$-classes}.)}
\end{enumerate}
We have conjectured that these conditions are also necessary \cite{HT01}.  
The main challenge in proving this is to show that 
the divisors $\rho$ described above obstruct line bundles from being ample.
For example, we expect (extremal) $(-10)$-classes $\rho$ to be Poincar\'e
dual to multiples of lines contained in planes $P\subset F$.  
The presence of such planes has implications for the birational
geometry of $F$, as we can take the {\em Mukai flop} 
or elementary transformation along $P$ \cite[0.7]{Mu84}
$$\begin{array}{rcccl}
   & & \mathrm{Bl}_PF\simeq \mathrm{Bl}_{P'}F' & & \\
  &\swarrow & & \searrow & \\
F & & & & F'.
\end{array}
$$
Indeed, since $P$ is Lagrangian, 
$\cN_{P/F} \simeq \Omega^1_P$ so the exceptional divisor 
$E\subset \mathrm{Bl}_PF$ is isomorphic to 
$\bP(\Omega^1_P)$.  This admits two $\bP^1$-bundle structures over
$\bP^2$, so we can blow down $E$ to obtain a nonsingular variety $F'$
birational to $F$.  This is also an irreducible holomorphic symplectic
variety, deformation equivalent to $F$ \cite[3.4]{Hu97}.

One especially
interesting case is when there are no square-$(0)$ or $(-2)$-classes 
(i.e., divisors $\rho$ with
$\left(\rho,\rho\right)=0,-2$) but multiple $(-10)$-classes.  Here
the nef cones of birational models of $F$
should be completely controlled by $(-10)$-classes.
Not only are the integral extremal rays of $F$ Poincar\'e
dual to $(-10)$-classes,
but this remains true for Mukai flops of $F$.
In this situation, we expect
$F$ to admit {\em infinite} sequences of Mukai flops.
However, Morrison \cite{Mor} and Kawamata \cite{Kaw} have conjectured the following:
\begin{conj}[Finiteness of models]
Let $F$ be a (simply-connected) Calabi-Yau manifold.  Then 
there are finitely-many minimal models of $F$ up to isomorphism.  
\end{conj}
How can this be reconciled with the existence of infinite sequences of flops?
The only possibility is that after a {\em finite} sequence of flops of $F$, we
arrive at a variety isomorphic to $F$.  This gives rise to birational maps 
$F\dashrightarrow F$ that are not automorphisms.  
These in turn act on $H^2(F,\bZ)$,
preserving the cone of moving divisors but not the nef cone.  

More specifically, consider a general cubic fourfold $X$ containing a cubic scroll. 
The Picard lattice ${\rm Pic}(F)$ has rank two, the associated quadratic form represents
$-10$ but not $-2$ or $0$. For such fourfolds we
\begin{itemize}
\item compute the ample and moving cone in $\Pic(F)$;
\item prove that $F$ does not admit biregular automorphisms;
\item exhibit a birational automorphism of infinite order explaining 
the chamber decomposition of the moving cone.
\end{itemize}
Our principal results are Theorems~\ref{theo:cubicscroll} and \ref{theo:main}.
The first exhibits explicit birational involutions on $F$ and factors
their indeterminacy.  The second describes the action of the birational
automorphism group on $H^2(F,\bZ)$.  

Miles Reid \cite[6.8]{Reid} has offered examples of elliptically-fibered threefolds
with an infinite number of distinct minimal models.  Morrison \cite{Mor} and Kawamata \cite{Kaw}
have proven finiteness results (up to isomorphism!) for
Calabi-Yau fiber spaces $F \ra S$ where $0<\mathrm{dim}(S)\le \mathrm{dim}(X)\le 3$.  
The case of Calabi-Yau manifolds of dimension $\ge 3$ remains open.  

The paper is organized as follows:  The first half is devoted to classical
results on cubic hypersurfaces.  
In Section~\ref{sect:cubic-3} we analyze cubic threefolds $Y$
with six ordinary double points 
in general position and their varieties of lines $F(Y)$.
Section~\ref{sect:deter} establishes a dictionary between determinantal
cubic surfaces and threefolds.
Section~\ref{sect:apply} develops this to explain the geometric 
properties of $Y$, e.g., a 
transparent description of the components of $F(Y)$  and how they
are glued together.  Finally, Section~\ref{sect:CTAD} completes the
correspondence between cubic threefolds with six double points
and determinantal cubic threefolds.  The second half focuses on
applications to the birational geometry of certain irreducible
holomorphic symplectic varieties.  
Section~\ref{sect:example} uses this information to construct birational
involutions on the variety of lines $F$ on a cubic fourfold containing $Y$.
In Section~\ref{sect:appl-cones} we explain the connection to our
conjecture on nef cones.  We close with an application to Zariski-density 
of rational points on $F$.  

\

Throughout, the base field is algebraically closed of characteristic zero.

\

{\bf Acknowledgments:} We are grateful to J\'anos Koll\'ar and Nick Shepherd-Barron
for useful conversations, and to Igor Dolgachev and Claire Voisin
for helpful comments.  
The first author was partially supported by National Science Foundation Grants
0554491 and 0134259.
The second author was partially supported by National Science Foundation
Grants 0554280 and 0602333.

\section{Cubic threefolds with six double points}
\label{sect:cubic-3}
We assume that $Y\subset \bP^4$ is a cubic hypersurface with
ordinary double points at $p_1,\ldots,p_6$, which are in linear general
position.  

\begin{lemm} \label{lemm:noplane}
The cubic hypersurface
$Y$ contains no planes and the variety of lines $F(Y)$ 
has the expected dimension two.  
\end{lemm}
\begin{proof}
Let $Y'$ denote a cubic threefold containing the plane
$$\Pi=\{x_0=x_1=0 \}.$$
Suppose $G$ is a homogeneous cubic equation for $Y'$ then
$$G=x_0Q_0+x_1Q_1$$
for quadratic forms $Q_0$ and $Q_1$.  The singular locus of $Y'$ contains
the subscheme defined by
$$x_0=x_1=Q_0=Q_1=0$$
which consists of four coplanar points.  Thus the singularities of
$Y'$ are not in linear general position.

Suppose that $F(Y)$ has dimension $>2$.  As the singularities of $Y$ are
ordinary double points, there is at most a one-parameter family of lines
through each singularity.  Thus the generic line of $F(Y)$ is contained in 
a smooth hyperplane section of $Y$.  Consider the incidence correspondence
$$Z=\{(\ell,H): \ell \subset Y, \ell \subset H \} \subset 
\mathrm{Gr}(2,\Lambda^{\perp})  \times \mathrm{Gr}(4,\Lambda^{\perp})$$
in the partial flag variety of $\Lambda^{\perp}$.  Since $Z$ has dimension five
the fibers of projection onto the second factor have dimension one, which
is impossible as smooth cubic surfaces have a finite number of lines.
\end{proof}

\begin{prop} \label{prop:project}
Let $Y$ be a cubic hypersurface with six ordinary double points 
$p_1,\ldots,p_6$ in linear general position.  
Projection from the point $p_6$
$$Y \dashrightarrow \bP^3$$
factors 
$$\begin{array}{rcl}
\tilde{Y}:=\Bl_{p_6}Y & \stackrel{\gamma}{\ra} & \bP^3 \\
 {\scriptstyle \delta}\downarrow  \ & & \\
Y  & &
\end{array}
$$
where $\delta$ is the blow up of a complete intersection $C_6$ of a 
smooth quadric and a cubic in $\bP^3$, consisting of two twisted
cubic curves meeting in five nodes.  
\end{prop}
\begin{proof}
The morphism $\gamma$ blows down all the lines in $Y$ incident to $p_6$;
since $p_6$ is an ordinary double point, these are parametrized by a 
complete intersection $C_6$ of a smooth quadric (the projectivized
tangent cone of $Y$ at $p_6$) and a cubic in $\bP^3$.
Furthermore, an easy computation using the Jacobian criterion shows that
$C_6$ is smooth except at the points $n_i=\gamma(p_i), i=1,\ldots,5$.
Note that $n_i$ corresponds to the line $\ell(p_i,p_6)$ joining $p_i$
to $p_6$.  

We claim that
$C_6$ has two irreducible components $E_6$ and $E^{\vee}_6$, each smooth and rational
of degree three, and $n_1,\ldots,n_5$ are nodes of $C_6$.  
Since the normalization of
$C$ has genus $-1$ it is necessarily reducible.
Consider the alternatives for the combinatorics of 
components:  If $C$ were to contain a component of degree one 
then this would meet the rest of $C$ in three nodes, say $n_1,n_2,n_3$.  
Then the ordinary double points $\{p_6,p_1,p_2,p_3 \} \in Y$
would all lie in a plane, contradicting our general position hypothesis.
If $C$ were to contain a component of degree two then this would meet
the rest of $C$ in four coplanar nodes, say $n_1,n_2,n_3,n_4$.  
Then $\{p_6,p_1,p_2,p_3,p_4 \} \in Y$ would span a three-dimensional
space, again contradicting our hypothesis.  
If $C$ contained a component
of degree three and arithmetic genus one (i.e., a nodal plane cubic) then
the quadric $Q$ would be degenerate.  
\end{proof}

\begin{rema}
This analysis implies that
$$n_1,n_2,n_3,n_4,n_5 \in Q \simeq \bP^1 \times \bP^1$$
satisfy the following genericity conditions:
\begin{itemize}
\item{the $n_i$ are distinct;}
\item{no two of the $n_i$ lie on a ruling of $Q$;}
\item{no four of the $n_i$ lie on a hyperplane
section of $Q \subset \bP^3$.}
\end{itemize}
Hence $S=\Bl_{n_1,\ldots,n_5}Q$
is isomorphic to a nonsingular cubic surface.
While we will not prove this, $S$ is isomorphic to the cubic surface
constructed from $Y$ in (\ref{eqn:cubictosegre})
of Section~\ref{sect:CTAD}.
In particular, $S$ does not depend on which double point $p_i \in Y$
we choose for our projection.  
\end{rema}

\begin{coro} \label{coro:linessing}
The singular locus of $F(Y)^{sing} \subset 
F(Y)$ is equal to the lines meeting the
singular points $p_1,\ldots,p_6 \in Y$.  The irreducible components of
$F(Y)^{sing}$ consist of twelve smooth rational curves 
$$E_1,E_1^{\vee},\ldots,E_6,E_6^{\vee},$$
where $E_j \cup E_j^{\vee}$ parametrizes the lines through $p_j$.  
The singularities of $F(Y)^{sing}$ are the 15 lines $\ell(p_i,p_j)$
joining singularities of $Y$, and 
$$\ell(p_i,p_j)=E_i \cap E_j = E_i \cap E_j^{\vee}= E_i^{\vee}\cap E_j=E_i^{\vee}\cap E_j^{\vee}.$$
\end{coro}
\begin{proof}
It is a general fact \cite[\S 1]{AK} that for any cubic hypersurface 
$Y'$, the variety of lines $F(Y')$ is smooth at lines avoiding the 
singularities of $Y'$.  Moreover, $F(Y')$ is singular at lines
passing through an ordinary double point of $Y'$ \cite[7.8]{CG}.  
The structure of the singular locus then follows from Proposition~\ref{prop:project}.  
\end{proof}

\begin{coro} \label{coro:unique}
The pair $(Y,p_6)$ is uniquely determined up to isomorphism by the
isomorphism class of the nodal curve $C_6$.  
\end{coro}
\begin{proof}
The curve $C_6$ is a stable curve of genus four and 
$C_6 \hookrightarrow \bP^3$ is its canonical embedding.  We can
characterize $Y$ as
the image of $\bP^3$ under the linear series of cubics passing through
$C_6$.  
\end{proof}

\section{Determinantal cubic surfaces and threefolds}
\label{sect:deter}

We review determinantal representations of smooth cubic surfaces.
The story begins with Grassmann \cite{Grass}  
who showed that cubic surfaces arise as the common points of three
nets of planes in $\bP^3$, i.e., the locus where a $3\times 3$ matrix
of linear forms on $\bP^3$ has nontrivial kernel.  Schr\"oter \cite{Schr}
showed that a generic surface admits such a realization 
and Clebsch \cite{Clebs}
tied these representations to the structure of the lines on the
cubic surface.  Dickson \cite{Dickson} addressed the problem of
expressing arbitrary smooth cubic surfaces in determinantal form.
See \cite[6.4]{Beau99} \cite{BK} for modern accounts and \cite{Dolg}
for further historical discussion.

\begin{prop} \label{prop:Grassmann}
Let $S\subset \bP^3$ be a smooth cubic surface.  Then there exists a $3\times 3$
matrix $M=(m_{ij})$ with entries linear forms on $\bP^3$ such that
$$S=\{\det(M)=0 \}.$$
Up to the left/right action of $\GL_3 \times \GL_3$, there are $72$ such representations, corresponding to sextuples of disjoint lines on $S$.
\end{prop}
This was extended by B. Segre \cite{BSegre} (cf. \cite[6.5]{Beau99})
to smooth cubic surfaces defined over arbitrary fields:
\begin{prop}
Let $S$ be a smooth cubic surface defined over an arbitrary field $k$.
Then the following conditions are equivalent:
\begin{itemize}
\item{There exists a $3\times 3$ matrix of linear forms over $k$ such that
$S=\{\det(M)=0 \}$.}
\item{$S$ contains a rational point and a sextuple
of disjoint lines defined over $k$.}
\item{$S$ admits a birational morphism to $\bP^2$ defined over $k$.}
\end{itemize}
\end{prop}
\noindent We emphasize that each individual line in the sextuple need not be defined
over $k$.  

C. Segre \cite[\S 12-14]{Segre} analyzed determinantal representations of cubic threefolds:
\begin{prop} \label{prop:Segre}
Let $Y \subset \bP^4$ be a generic cubic hypersurface realized as the determinant of a $3\times 3$
matrix of linear forms.  Then $Y$ has six ordinary double points, in linear general position.
Conversely, any cubic hypersurface with six ordinary double point in linear general position is
determinantal.  
\end{prop}
\noindent For completeness, we will provide an argument in
Propositions~\ref{prop:transversal} and \ref{prop:inverse}.

Our main goal is to explain how all these classical theorems are related.
Here is the key geometric ingredient:  Let $W$ be a vector space with a nondegenerate
quadratic form $\left(,\right)$;  taking orthogonal complements, we obtain a natural identification
\begin{equation} \label{eq:perp}
\begin{array}{rcl}
\Gr(n,W)&=&\Gr(\dim(W)-n,W) \\
\Lambda & \mapsto & \Lambda^{\perp}. 
\end{array}
\end{equation}
Let $G$ be a group acting linearly on $W$, 
with the natural induced action
on $\Gr(n,W)$ and the action on 
$\Gr(\dim(W)-n,W)$ induced by (\ref{eq:perp}).

We are especially interested in the case where $W=\End(V)$ 
for some vector space $V$ of dimension $d$,
$\left(,\right)$ is the trace pairing
$$\left(A,B\right)=\mathrm{tr}(AB),$$
and $\GL(V)\times \GL(V)$ acts on $\End(V)$ by left-right multiplication
$$(g_1,g_2)\cdot M=g_1Mg_2^t.$$
Here we are using the transpose operator
$$t:\End(V) \ra \End(V)$$
obtained by dualizing and then applying the trace pairing.
Consider the semidirect product
\begin{equation} \label{eqn:semidirect}
1 \ra \GL(V) \times \GL(V)  \ra G \ra \fS_2 \ra 1
\end{equation}
where $\fS_2$ acts by exchanging the factors.  
Since 
$$(g_1,g_2) \cdot M^t=g_1M^tg_2^t=((g_2,g_1)\cdot M)^t$$
$G$ also acts naturally on $\End(V)$ and thus on the Grassmannians
$$\Gr(n,\End(V)) = \Gr(d^2-n,\End(V)).$$

Consider the rank stratification on $\End(V)$
$$0 \subset \Sigma_1 \subset \Sigma_2 \ldots \subset \Sigma_{d-1} \subset \End(V)$$
which is invariant under the group actions.  
Recall the description of the tangent space of $\Sigma_k$ at a matrix $A$ of rank $k$ 
(see, for instance, \cite[pp. 68]{ACGH}):
$$T_A\Sigma_k=\{M \in \End(V): M(\mathrm{ker}A) \subset \mathrm{im}A \}.$$
We can then express
$$(T_A\Sigma_k)^{\perp}=\{N \in \End(V): NA=AN=0 \},$$
which is a linear subspace contained in $\Sigma_{d-k}$.  
If $A$ has rank $<k$ then $T_A\Sigma_k=\End(V)$;  furthermore,
$\Sigma_{k-1}$ is the singular locus of $\Sigma_k$ \cite[pp. 69]{ACGH}.

\begin{prop} \label{prop:linalg}
Assume $V$ is three-dimensional and $\Lambda \subset \End(V)$
is four-dimensional.   The following conditions are equivalent:
\begin{itemize}
\item{
$\Lambda$ is tangent to $\Sigma_2$ at a smooth point or intersects $\Sigma_1$ nontrivially;}
\item{$\Lambda^{\perp}$ is tangent to $\Sigma_1$ or is
tangent to $\Sigma_2$ at a smooth point.}
\end{itemize}
\end{prop}
\begin{proof}
Suppose that $\Lambda$ is tangent to $\Sigma_2$ at a rank-two matrix $A_0$, i.e.,
$$\Lambda \subset T_{A_0}\Sigma_2=
\{M\in \End(V): M(\mathrm{ker}(A_0)) \subset \mathrm{im}(A_0)\}.$$
Let $B_0$ be a matrix with $\mathrm{ker}(B_0)=\mathrm{im}(A_0)$ 
and $\mathrm{im}(B_0)=\mathrm{ker}(A_0)$;  this is unique up to scalars.
Since $B_0^{\perp}=T_{A_0}\Sigma_2$ we have
$B_0 \in \Lambda^{\perp}$.  
Consider the induced linear transformation
$$\Lambda \ra \Hom(\mathrm{im}(A_0),V/\mathrm{im}(A_0)) \oplus \Hom(\mathrm{ker}(A_0),\mathrm{im}(A_0))$$
with kernel containing $A_0$.  The image has dimension at most three, so there exists a matrix
$B_1$ linearly independent from $B_0$ such that 
$$B_1(\mathrm{im}(A_0))\subset \mathrm{ker}(A_0), B_1(\mathrm{ker}(A_0)) \subset \mathrm{im}(A_0)$$
and 
$$\mathrm{tr}(AB_1)=0 \text{ for each }A\in \Lambda.$$
Thus $B_1 \in T_{B_0}\Sigma_1 \cap \Lambda^{\perp}.$

Now suppose $\Lambda$ is incident to $\Sigma_1$ at a rank-one matrix $A_0$;  again, we have
$$A_0^{\perp}=\{M:M(\mathrm{im}(A_0)) \subset \mathrm{ker}(A_0) \}.$$
Consider the subspace of dimension four
$$\{M: M(\mathrm{im}(A_0))=0, M(\mathrm{ker}(A_0))\subset \mathrm{ker}(A_0)\}.$$ 
This necessarily meets $\Lambda^{\perp}$ in a nontrivial matrix $B$.  If
$B$ has rank two then $\mathrm{im}(B)=\mathrm{ker}(A_0)$ and 
$\mathrm{ker}(B)=\mathrm{im}(A_0)$, thus
$$
T_B\Sigma_2=\{M:M(\mathrm{ker}(B))\subset \mathrm{im}(B) \}=A_0^{\perp}.$$
Since $A_0^{\perp} \supset \Lambda^{\perp}$ we have
$$T_B\Sigma_2 \supset \Lambda^{\perp}$$
whence $\Lambda^{\perp}$ is tangent to $\Sigma_2$ at the smooth point $B$.  
If $B$ has rank one then the tangent space
$$T_B\Sigma_1\subset A_0^{\perp}$$
has dimension five, so
$$T_B\Sigma_1 \cap \Lambda^{\perp}  \subset A_0^{\perp}$$
has dimension at least two, i.e., $\Sigma_1$ is tangent to $\Lambda^{\perp}$
at $B$.  

Conversely, suppose that $\Lambda^{\perp}$ is tangent to $\Sigma_1$ at $B_0$. 
Choose $B_1$ linearly independent from $B_0$ such that
$$B_1 \in \Lambda^{\perp} \cap T_{B_0}\Sigma_1$$
which means that $B_1(\mathrm{ker}(B_0))\subset \mathrm{im}(B_0)$.
Consider the four-dimensional vector space of matrices
$$W=\{M: M(\mathrm{ker}(B_0)) \subset \mathrm{ker}(B_0), M(\mathrm{im}(B_0))=0 \}.$$
These matrices are contained in $\{B_0,B_1\}^{\perp}$.  Since $\Lambda$ has
dimension four
$$W \cap \Lambda \ \subset \ \{B_0,B_1 \}^{\perp}$$
is nontrivial.  Choose $A \neq 0$ in this intersection.  If $A$ is of rank one then we are done.
Otherwise, the rank of $A$ equals two.  However, we have
$$T_A\Sigma_2=\{M:M(\mathrm{ker}(A))\subset \mathrm{im}(A) \}=
		\{M:M(\mathrm{ker}(B_0)) \subset \mathrm{ker}(B_0) \}=B_0^{\perp}.$$
Thus $T_A\Sigma_2$ contains $\Lambda$, i.e., $\Lambda$ meets $\Sigma_2$
nontransversally at $A$.  

Now suppose $\Lambda^{\perp}$ is tangent to $\Sigma_2$ at a smooth point $B$,
which has rank two.  Thus 
$$\Lambda^{\perp} \subset \{ M: M(\mathrm{ker}(B))\subset \mathrm{im}(B) \}$$
and 
$$\Lambda \supset A_0$$
where $A_0$ is a nonzero matrix with $\mathrm{ker}(A)=\mathrm{im}(B)$ and
$\mathrm{im}(A)=\mathrm{ker}(B)$.  Any such matrix has rank one, hence $\Lambda$
meets $\Sigma_1$ nontrivially.
\end{proof}

We use this to interpret our determinantal expressions for
cubic hypersurfaces.  
Tensor multiplication gives the Segre embedding 
$$\bP(V) \times \bP(V^{\vee}) \hookrightarrow \bP(\End(V));$$
the image has degree six and may be identified with $\Sigma_1$. 
Given a four-dimensional subspace
$$\Lambda \subset \End(V)$$
the intersection 
\begin{equation} \label{eq:defS}
S:=\bP(\Lambda) \cap \Sigma_2 \subset \bP(\Lambda)\simeq \bP^3
\end{equation}
is a determinantal cubic surface.  It is smooth precisely when $\Lambda$
meets $\Sigma_2$ transversely at smooth points.
Then we obtain an embedding
$$\begin{array}{rcl}
\iota: S & \hookrightarrow & \bP(V) \times \bP(V^{\vee}) \\
       s & \mapsto & (\mathrm{ker}(s),\mathrm{im}(s))
\end{array}
$$
such that the projections induce the blow-up realizations of $S$ 
(cf. Proposition~\ref{prop:Grassmann})
$$\beta:S \ra \bP(V), \quad \beta^{\vee}:S \ra \bP(V^{\vee}).$$

Let $\Lambda^{\perp}$ be the orthogonal complement to $\Lambda$ with respect to the trace pairing.
Then 
\begin{equation} \label{eq:defY}
Y:=\bP(\Lambda^{\perp}) \cap \Sigma_2 \subset \bP(\Lambda^{\perp}) \simeq \bP^4
\end{equation}
is a determinantal cubic threefold.  It is necessarily singular along the 
points of $\bP(\Lambda^{\perp}) \cap \Sigma_1$.  If $\bP(\Lambda^{\perp})$ intersects 
$\Sigma_1$ and the smooth points of $\Sigma_2$ transversely then the Bezout theorem
implies that the singular locus of $Y$ is
$$\{p_1,\ldots,p_6 \}:= \Sigma_1 \cap \bP(\Lambda^{\perp}).$$ 
Note that these give a sextuple of points in $\bP(V)\times \bP(V^{\vee})\simeq \bP^2 \times \bP^2$;
a straightforward cohomology computation shows these are in linear general position in $\bP(\End(V))$.  

\begin{prop} \label{prop:transversal}
Let $S$ and $Y$ be determinantal cubic hypersurfaces defined by
Equations~\ref{eq:defS} and \ref{eq:defY} above.  Then
$Y$ is a cubic threefold with six ordinary double points in linear general position if
and only if $S$ is a smooth cubic surface. 
We thus obtain a identification
$$\left\{ \begin{array}{c}
\text{determinantal cubic } \\
\text{threefolds with six } \\
\text{ordinary double points} \\
\text{in linear general position } 
\end{array} \right\} 
=
\left\{ \begin{array}{c}
\text{determinantal cubic surfaces } \\
\text{without singularities }  \\
\end{array} \right\}
$$
that is equivariant with respect to the action of $G$.
\end{prop}
Indeed, Proposition~\ref{prop:linalg} says we can
identify the open subsets in
$$\Gr(4,\End(V)) = \Gr(5,\End(V))$$
where our transversality conditions hold.

\section{Geometric applications of the determinantal description}
\label{sect:apply}

In this section, we assume that $S$ and $Y$ satisfy the conclusions of 
Proposition~\ref{prop:transversal}.
The determinantal description allows a transparent derivation of many 
of the key properties of $Y$.

\begin{prop} \label{prop:getlines}
Let $F(Y)$ denote the variety of lines on $Y$.  We have a natural
surjective morphism
$$\nu: \bP(V) \sqcup S \sqcup \bP(V^{\vee}) \ra F(Y)$$
that maps each component birationally onto its image.  
\end{prop}
\begin{proof}
For each point $[v] \in \bP(V)$, let 
$$\ell_{[v]}=\{y=[\phi]: \phi(v)=0 \}=
\{y=[\phi]: v \in \mathrm{ker}(\phi) \}
\subset Y$$
where $\phi \in \End(V)$ represents $y$.  
This is a linear subspace of codimension at most three
in $\bP(\Lambda^{\perp})$.  Indeed, elements of
$$\Lambda^{\perp} \cap \{M: v \subset \mathrm{ker}(M) \}$$
automatically have vanishing determinants.  Lemma~\ref{lemm:noplane}
guarantees $Y$ does not contain any planes, so we  
conclude that $\ell_{[v]}$
is a line.  

Similarly, for $[v^{\vee}] \in \bP(V^{\vee})$ we also get lines 
$$\ell_{[v^{\vee}]}=\{y=[\phi]: v^{\vee}\circ \phi =0 \}=
\{y=[\phi]: v^{\vee} \in \mathrm{ker}(\phi^t) \}   \subset Y.$$

Given $s=[\sigma]\in S$ with $\sigma \in \Lambda$,
we have the locus
$$\ell_s=\{y=[\phi]: \sigma \phi \sigma = 0 \}\subset Y.$$
Since $\sigma$ has rank two, this condition translates into the vanishing of the
$2\times 2$ matrix of the induced map
$$\mathrm{im}(\sigma) \stackrel{\phi}{\ra} V/\mathrm{ker}(\sigma).$$  
However, the orthogonality assumption 
$\mathrm{tr}(\sigma \phi)=0$ implies that 
there are only three independent linear
conditions.  In particular, $\ell_s$ is a line in $Y$.  

Combining these three constructions, we obtain the morphism $\nu$.  
We next show that $\nu$ is surjective.  Lemma~\ref{lemm:noplane}
implies that $F(Y)$ is two-dimensional.  A standard intersection theory
computation \cite[14.7.13]{Fulton} shows that $\deg F(Y)=45$ (with
respect to the Pl\"ucker embedding of the Grassmannian).  
However, we can compute the pull back
$$\nu^*\cO_{F(Y)}(1)=(\cO_{\bP(V)}(3),\cO_S(3),\cO_{\bP(V^{\vee})}(3))$$
which means that
$$\deg(\bP(V))=\deg(\bP(V^{\vee}))=9, \quad \deg(S)=27.$$
Thus all the components of $F(Y)$ are in the image of $\nu$;
furthermore, $\nu$ maps each component birationally onto its image.
\end{proof}

\begin{coro} \label{coro:fibration}
Retain the notation of Proposition~\ref{prop:getlines} and
let $y\in Y$ be a nonsingular point.  The 
components of $F(Y)$ dominated by $\bP(V)$ and $\bP(V^{\vee})$ 
each admit a unique line passing through $y$.  The component
dominated by $S$ admits four lines passing through $y$.
\end{coro}
\begin{proof}
The first statement is easily verified using linear algebra.  We second
can be deduced from the fact that a generic $y\in Y$ lies on
six lines in $Y$.  
\end{proof}

\begin{prop} \label{prop:components}
The morphism 
$$\nu: \bP(V) \sqcup S \sqcup \bP(V^{\vee}) \ra F(Y)$$
induces the following identifications:
Consider the distinguished double-six on $S$
$$\{E_1,\ldots,E_6;E_1^{\vee},\ldots,E_6^{\vee}\},$$
with each 6-tuple blowing down to a collections of points
$$\{q_1,\ldots,q_6 \} \subset  \bP(V), \quad
\{q^{\vee}_1,\ldots,q^{\vee}_6 \} \subset  \bP(V^{\vee}).$$
Let $\{\ell_1,\ldots,\ell_6\}$  and $\{\ell^{\vee}_1,\ldots,\ell^{\vee}_6\}$
be the lines in $\bP(V^{\vee})$ and $\bP(V)$ dual to these points.  
We have natural isomorphisms
for each $i$:
$$\begin{array}{rl}
\psi_i:\ell_i \stackrel{\sim}{\ra} E_i, &
\psi_i(\ell_i\cap \ell_j)=E_i\cap E_j^{\vee}\subset S, \\
\psi^{\vee}_i:\ell_i^{\vee}\stackrel{\sim}{\ra} E^{\vee}_i, &
\psi_i(\ell^{\vee}_i\cap \ell^{\vee}_j)=E^{\vee}_i\cap E_j\subset S.
\end{array}$$
\end{prop}
\begin{proof}
We break up the argument into a sequence of lemmas:
\begin{lemm} \label{lemm:identify}
The morphism
$\nu$ maps $E_i,E_i^{\vee}\subset S,\ell_i\subset \bP(V^{\vee})$
and $\ell_i^{\vee} \subset \bP(V)$ to the locus $C_i$ of lines passing
through $p_i$.  Furthermore, $E_i$ and $\ell_i$ parametrize $y=[\phi]\in Y$ such
that $\mathrm{im}(\phi)\supset \mathrm{im}(p_i)$;  $E_i^{\vee}$ and $\ell_i^{\vee}$
parametrize $y=[\phi]$ such that $\mathrm{ker}(\phi) \subset \mathrm{ker}(p_i)$.
Here we regard the singularity $p_i \in Y$ as an element
$\Lambda^{\perp} \cap \Sigma_1.$
\end{lemm}
\begin{proof}
The determinantal description of $S$ identifies
$$E_i=\{s\in\Lambda: \mathrm{ker}(s)=\mathrm{im}(p_i) \}.$$
Similarly, we have
$$E^{\vee}_i=\{s\in\Lambda: \mathrm{im}(s)=\mathrm{ker}(p_i) \}.$$
On the other hand,
$$\ell_i=\{\mathrm{im}(s):
 s\in \Lambda \text{ with } \mathrm{ker}(s)=\mathrm{im}(p_i) \} \subset \bP(V^{\vee})$$
and 
$$\ell^{\vee}_i=\{\mathrm{\ker}(s)):
 s\in \Lambda \text{ with }\mathrm{im}(s)=\mathrm{ker}(p_i) \} \subset \bP(V).$$

Thus for $s_i=[\sigma] \in E_i$
$$\nu(s_i)=[\{y = [\phi]: \phi(\mathrm{im}(\sigma))\subset \mathrm{ker}(\sigma)=\mathrm{im}(p_i) \}]$$
which is a line through $p_i$.  
On the other hand, for $v^{\vee}_i \in \ell_i$ (where $v^{\vee}_i \in V^{\vee}$ satisfies $v^{\vee}_i(\mathrm{im}(p_i))=0$)
we have
$$\nu(v^{\vee}_i)=[\{y=[\phi]: v_i^{\vee}(\mathrm{im}(\phi))=0 \}]$$
which is also a line through $p_i$.  As we vary $s_i \in E_i $ and $v^{\vee}_i \in \ell_i$, we get the locus of $y=[\phi]$
such that $\mathrm{im}(\phi)\supset \mathrm{im}(p_i)$.  

The analogous statements for $E_i^{\vee}$ and $\ell_i^{\vee}$ are proven similarly.  
\end{proof}

There is an obvious identification
\begin{equation} \label{eq:obvious}
\begin{array}{c}
\ell_i=\bP(q_i^{\perp})=\bP((V/q_i)^{\vee})=\bP(V/q_i)
=\bP(\Hom(q_i,V/q_i))  \\
=\bP(T_{q_i}\bP(V))=E_i; 
\end{array}
\end{equation}
note that if $W$ is a two-dimensional vector space then the isomorphism
$W=W^{\vee}\otimes \bigwedge^2 W$ induces a natural isomorphism $\bP(W)=\bP(W^{\vee})$.  
This is {\em not} the gluing inducing $\nu$.  
However, note that this takes the points $\ell_i \cap \ell_j$ to the
intersections $E_i \cap \fl_{ij}$, 
where $\fl_{ij}$ is the
proper transform of the line joining $q_i$ and $q_j$.  
Using (\ref{eq:obvious}), it suffices to express
$\psi_i$ and $\psi_i^{\vee}$ as automorphisms of $E_i$ and $E_i^{\vee}$.  

The gluings $\ell_i\simeq E_i$ and $\ell_i^{\vee}\simeq E_i^{\vee}$ 
will be obtained from the following:  
\begin{lemm} \label{lemm:glue}
There exists a projectivity $\psi_i:E_i\ra E_i$ mapping 
$E_i \cap \fl_{ij}$ to $E_i \cap E_j^{\vee}$ for each $j\neq i$.  
The analogous
statement holds for $E_i^{\vee}$.
\end{lemm}
\begin{proof}
For notational simplicity we take $i=1$.  Consider the conic
bundle $S\ra \bP^1$ given by the pencil of cubics on $\bP(V)$
double at $q_1$ and containing $q_2,\ldots,q_6$.  The degenerate
fibers are
$$\fl_{1j} \cup E_j^{\vee}, j=2,\ldots,6.$$
The curve $E_1$ is a bisection of this conic bundle, so there
is a covering involution $\psi_1:E_1 \ra E_1$ taking 
$\fl_{1j} \cap E_1$ to $E_j^{\vee} \cap E_1$.  
\end{proof}

It remains to check that this is in fact the identification
induced by $\nu$.  However, we know from Corollary~\ref{coro:linessing} 
that $\nu$ glues the points $E_i\cap E_j^{\vee}$, $E_i^{\vee}\cap E_j$,
to $[\ell(p_i,p_j)]$, the line in $Y$ joining $p_i$ and $p_j$.  
Now $\ell_i$ and $\ell_j$ meet in $\bP(V^{\vee})$, and $\ell_i^{\vee}$
and $\ell_j^{\vee}$ meet in $\bP(V)$;  thus these points must also
be mapped by $\nu$ to $[\ell(p_i,p_j)]$.  In general, the
isomorphisms $\psi_i$ and $\psi^{\vee}_i$ are the unique ones
identifying all these points.  
\end{proof}

\begin{prop}
The morphism $\nu$ is obtained by gluing $\bP(V),S,$
and $\bP(V^{\vee})$ using the identifications described in
Proposition~\ref{prop:components}.
\end{prop}
\begin{proof}
Let $F'$ denote the surface obtained by gluing $\bP(V),S$, and
$\bP(V^{\vee})$ using the identifications $\psi_i$ and $\psi_i^{\vee}$.  
Again, $F'$ contains twelve distinguished rational curves
$$\ell_i=E_i,\ell_i^{\vee}=E_i, \ i=1,\ldots,6$$
and {\em fifteen} distinguished points
$$\ell_i\cap \ell_j=E_i \cap E_j^{\vee}=\ell_i^{\vee}\cap \ell^{\vee}_j=E_i^{\vee} \cap E_j,$$
which map surjectively onto $F(Y)^{sing}$ (by Corollary~\ref{coro:linessing}).  

We have already seen that $\nu$ factors through $F'$;  it only remains
to prove that the induced morphism $F' \ra F(Y)$ is an isomorphism.
Our analysis of the gluings over $F(Y)^{sing}$ shows that
$\xi$ is a bijection over $F(Y)^{sing}$.  

We first check that $\nu$ is the normalization of $F(Y)$.  
Proposition~\ref{prop:getlines} shows that $\nu$ maps each irreducible
component birationally onto its image.  It follows that the restrictions
$$\bP(V) \ra \nu(\bP(V)), \quad \bP(V^{\vee}) \ra \nu(\bP(V^{\vee}))$$
are normalization maps.  Consider the factorization of $\nu|S$
through the normalization of its image
$$S \ra \nu(S)' \ra \nu(S);$$
this reverses the identifications induced by the $\psi_i$ and $\psi_i^{\vee}$.  
The images of the six lines
$E_1,\ldots,E_6,$ (and $E_1^{\vee},\ldots,E_6^{\vee}$) in 
$\nu(S)'$ are pairwise disjoint.  Hence $S\ra \nu(S)'$
contracts no curves and thus is an isomorphism;  $S$ is the
normalization of $\nu(S)$.  

This analysis implies $F' \ra F(Y)$ is bijective.  

The Fano scheme $F(Y)$ is defined by the degeneracy
locus of a vector bundle over the Grassmannian $\mathrm{Gr}(2,5)$,
with the expected dimension (by Lemma~\ref{lemm:noplane}).  Thus $F(Y)$
is a local complete intersection scheme and thus 
is Cohen-Macaulay;  hence it
has no embedded points and is seminormal.  The universal property
of seminormalization then implies that $F' \ra F(Y)$ is
an isomorphism.
\end{proof}

The determinantal description offers a transparent construction for the 
cubic scrolls on $Y$.  Each point $[v] \in \bP(V)$ determines a line
in $\bP(V^{\vee})$, which may be interpreted as a ruled surface
$T_v\subset Y$ using the analysis of the components of $F(Y)$ in 
Proposition~\ref{prop:getlines}:
\begin{prop} \label{prop:scroll}
For each $[v]\in \bP(V)$, the locus
$$T_v=\{ y \in Y: v \in \mathrm{im}(y) \}$$
is a cubic scroll.  The ruling arises from
the morphism
$$\begin{array}{rcl}
T_v & \ra & \bP(V/\mathrm{span}(v)) \\
 y & \mapsto & \mathrm{im}(y)
\end{array}$$
with fibers $\ell_{[v^{\vee}]}$, where $v^{\vee}\neq 0\in V^{\vee}$ 
with $v^{\vee}(v)=0$.  

Similarly, for $[v^{\vee}] \in \bP(V^{\vee})$ the locus
$$T_{v^{\vee}}=\{ y\in Y: v^{\vee}(\mathrm{ker}(y))=0 \}$$
is a cubic scroll.  Each union 
$$T_{v} \cup T_{v^{\vee}}=Y \cap Q$$
where $Q$ is a quadric hypersurface.  

If $s\in S$ and $\ell_s$ denotes the corresponding line in $Y$ then
$\ell_s\subset T_{\beta(s)}$ (resp. $T_{\beta^{\vee}(s)})$
is a section of the ruling.  
\end{prop}
\begin{proof}
Choose a basis $v,v',v''$ of $V$ such that $v^{\vee}(v')=v^{\vee}(v'')=0$.  
The closure of the locus of rank-two matrices with image
containing $v$ can be written
$$B=\left( \begin{matrix} b_{11} & b_{12} & b_{13} \\
		        b_{21} & b_{22} & b_{23} \\
			b_{31} & b_{32} & b_{33} \end{matrix} \right)$$
where the bottom two rows are linearly dependent.  This defines
a closed subset in $\bP(\End(V))$.  Geometrically, this is
a cone over the Segre embedding
$$\bP^1 \times \bP^2  \subset \bP^5$$
with a vertex a projective plane.  Intersecting this with $\Lambda^{\perp}$
yields a hyperplane section of $\bP^1 \times \bP^2$, which is a cubic scroll.

On the other hand, the closure of the locus of rank-two matrices with
kernel annihilated by $v^{\vee}$ are those whose right two columns are linearly
dependent.  The union of these two loci are given by the intersection
$$\{\det(B)=0\} \cap \{ b_{22}b_{33}-b_{23}b_{32}=0 \} \subset \bP(\mathrm{End}(V)),$$
i.e., the intersection of $\Sigma_2$ with a quadric hypersurface.  

Recall that $\ell_s$ was defined in the proof of Proposition~\ref{prop:getlines}
$$\ell_s=\{y=[\phi]:\sigma \phi \sigma =0 \}.$$
Fix a ruling in $T_{\beta(s)}$: 
Regarding $\beta(s)=\mathrm{ker}(\sigma)$ as a line in $V$, we choose
a two-dimensional subspace $\mathrm{ker}(\sigma) \subset U\subset V$,
and consider the matrices $\phi$ with image $U$.  
This imposes one additional linear constraint
on the matrix entries of $\phi$, so 
each ruling meets $\ell_s$ in one point.  

\end{proof}

\begin{prop} \label{prop:twisted}
For each $s\in S$, 
$$T_{\beta(s)} \cap T_{\beta^{\vee}(s)}= \ell_s \cup R_s$$
where $R_s\subset Y $ is a twisted quartic curve
passing through the singularities $p_1,\ldots,p_6$.  
\end{prop}
\begin{proof}
Generically, the cubic scrolls are nonsingular and isomorphic to 
$\bP^2$ blown up at one point, in which case $\ell_s \subset T_{\beta(s)}$
is the exceptional curve.  
Let $R_s$ denote the union of components of the intersection
other than $\ell_s$.  
We have shown that $T_{\beta{s}} \cup T_{\beta^{\vee}(s)}$ is a complete intersection
of a quadric and cubic in $\bP^4$, and thus is a singular K3 surface.  Adjunction
shows that
$$K_{T_{\beta(s)}}+ R_s+\ell_s\equiv 0;$$
hence $R_s$ has degree four and genus zero.  
\end{proof}

\section{Cubic threefolds with six double points are determinantal}
\label{sect:CTAD}

Here we complete C. Segre's determinantal construction of
cubic threefolds
with six double points:

\begin{prop} \label{prop:inverse}
Each cubic threefold with six ordinary double points
in linear general position is determinantal.  
\end{prop}
We prove Proposition~\ref{prop:inverse}
using the geometry of the twisted
quartic curves in a determinantal cubic threefold,
following \cite[3.2-3.4]{CLSS}. 
One key tool is the {\em Segre threefold}
$\fS \subset \bP^4$;  we recall its basic properties:
\begin{itemize}
\item{Given $p_1,\ldots,p_6\in \bP^4$ in linear general position,
the linear series of cubics double at these points induces a morphism
$$\varpi:\Bl_{p_1,\ldots,p_6}\bP^4 \ra \bP^4$$
with image $\fS$ and fibers twisted quartic curves containing
the points $p_1,\ldots,p_6$.
If we choose $p_1=[1,0,0,0,0]$,$p_2=[0,1,0,0,0]$, \\
$p_3=[0,0,1,0,0]$,
$p_4=[0,0,0,1,0]$,$p_5=[0,0,0,0,1]$, and $p_6=[1,1,1,1,1]$ then
the cubics double at these points are
$$\begin{array}{c}
y_0=(x_3-x_4)x_0(x_1-x_2), \quad 
y_1=(x_4-x_0)x_1(x_2-x_3), \\
y_2=(x_0-x_1)x_2(x_3-x_4), \quad
y_3=(x_1-x_2)x_3(x_4-x_1), \\
y_4=(x_2-x_3)x_4(x_0-x_1)
\end{array}
$$
which satisfy
$$y_0y_1y_2+y_1y_2y_3+y_2y_3y_4+y_3y_4y_0+y_4y_0y_1=0.$$
}
\item{$\fS$ contains $10$ ordinary double points and $15$ planes.}
\item{The nonsingular twisted quartic curves map to an open subset
of $\fS$ that is isomorphic to $\cM_{0,6}$, the moduli
space of genus-zero curves with six marked points.
The morphism $\varpi$ is the universal family over $\cM_{0,6}$.}
\item{The inclusion $\cM_{0,6} \hookrightarrow \fS$ extends
to an isomorphism \cite{Hu}
$$(\bP^1)^6 /\!\!/ \mathrm{SL}_2 \stackrel{\sim}{\ra} \fS$$
from the GIT quotient of six points in $\bP^1$ with the symmetric
linearization.}
\end{itemize}
Thus we have a morphism 
\begin{equation}
\left\{ \begin{array}{c}
\text{cubic threefolds with } \\
\text{ordinary double points } \\
\text{at $p_1,\ldots,p_6$} 
\end{array} \right\} \ra
\left\{ \begin{array}{c}
\text{smooth cubic surfaces }\\
\text{arising as hyperplane }\\
\text{sections of $\fS\subset \bP^4$}
\end{array} \right\}
 \label{eqn:cubictosegre}
\end{equation}
\begin{rema}
In general,
M. Kapranov \cite{Kap2} \cite[4.3]{Kap} has shown that $\cM_{0,n-1}$
can be identified with the rational normal curves in $\bP^{n-3}$
passing through points $p_1,\ldots,p_{n-1} \in \bP^{n-3}$ in
linear general position.  The rational normal curves are the universal
curve,  with $p_1,\ldots,p_{n-1}$ tracing out the marked points.
Identifying the universal curve over $\ocM_{0,n-1}$ with $\ocM_{0,n}$,
there is a morphism
$$\begin{array}{rcl}
\ocM_{0,n} & \ra & \bP^{n-3} \\
     (C,p_1,\ldots,p_n) & \mapsto & p_n,
\end{array}
$$
factoring through $\Bl_{p_1,\ldots,p_{n-1}}\bP^{n-3}$.  
\end{rema}

Recall our previous notation:  Let
$$\{E_1,\ldots,E_6;E_1^{\vee},\ldots,E_6^{\vee} \}$$
denote the double-six on $S$,
$\beta:S \ra \bP(V)$ and $\beta^{\vee}:S \ra \bP(V^{\vee})$
the associated contractions, and 
$$\{q_1,\ldots,q_6; q_1^{\vee},\ldots, q_6^{\vee} \}$$
the images of the exceptional divisors.  
There is an involution of the Picard lattice taking 
$E_i$ to $E_i^{\vee}$ for $i=1,\ldots,6$.  
If the Picard lattice is presented
$$\bZ L +\bZ E_1 + \ldots + \bZ E_6, \quad E_i^2=-1,E_iE_j=\delta_{ij},
                                                L^2=1, LE_i=0,
$$
the involution takes the form
\begin{equation}
E_i \mapsto 2L-E_j-E_k-E_a-E_b-E_c=E^{\vee}_i, \label{eqn:inv}
\end{equation}
where $\{i,j,k,a,b,c\}$ is a permutation of the indices $\{1,\ldots, 6\}$.

We shall need a version of Cremona's hexahedral construction \cite{Cr,Dolg}:
\begin{prop} \label{prop:tosegre}
Let $S^{\circ}\subset S$ denote the complement to the lines in $S$.
For $s\in S^{\circ}$ consider the images of
$q_1,\ldots,q_6$ and $q_1^{\vee},\ldots, q_6^{\vee}$
under the projections
\begin{equation} \label{eqn:project}
\bP(V) \dashrightarrow \bP(V/\beta(s)) \quad
\bP(V^{\vee}) \dashrightarrow \bP(V^{\vee}/\beta^{\vee}(s)),
\end{equation}
which determine elements $j(s),j^{\vee}(s) \in \cM_{0,6}$.
Then we have the following:
\begin{itemize}
\item{$j(s)=j^{\vee}(s)$ for each $s\in S^{\circ}$;}
\item{there exists an extension 
$j:S\ra \fS$;}
\item{the image of $j$ is a hyperplane section of $\fS \subset \bP^4$;}
\item{conversely, each smooth hyperplane section $S\subset \fS$ 
is a cubic surface with a distinguished ordered double-six.}
\end{itemize}
Thus we obtain an identification
$$
\left\{
\begin{array}{c}
\text{smooth cubic surfaces} \\
\text{with a double-six of}  \\
\text{ordered lines}
\end{array}  
\right\} 
\simeq
\left\{
\begin{array}{c}
\text{smooth hyperplane }\\
\text{sections of $\fS\subset \bP^4$}
\end{array}
\right\}. 
$$
\end{prop}
\begin{proof}
Fix $s\in S^{\circ}$
and consider the {\em degree-two} Del Pezzo surface
$S':=\mathrm{Bl}_s(S)=\mathrm{Bl}_{q_1,\ldots,q_7}(\bP(V))$.  
The projections (\ref{eqn:project}) induce conic bundles
$$\varphi:S' \ra \bP^1, \quad \varphi^{\vee}:S' \ra \bP^1$$
with degenerate fibers corresponding to the images of
$q_1,\ldots,q_6$ and $q_1^{\vee},\ldots,q_6^{\vee}$ respectively.
However, each degree-two Del Pezzo admits a canonical involution, i.e. 
the covering involution of the anticanonical morphism $S' \ra \bP^2$.
Moreover, $\varphi$ and $\varphi^{\vee}$ are conjugate under
this involution and thus have the same degenerate fibers.
We conclude that $j(s)=j^{\vee}(s)$ in $\cM_{0,6}$.  

We extend $j$ to $S$:  Assume first that $\beta(s)\neq q_1,\ldots,q_6$.
We still have a conic bundle $\varphi:S' \ra \bP^1$ but the
images of $q_i$ and $q_j$ in $\bP^1$ coincide if $\beta(s)\in \fl_{ij}$,
the line joining $q_i$ and $q_j$.  However, since no three of the
$q_i$ are collinear at most two points may coincide, so the image
of $(q_1,\ldots,q_6)$ is a GIT-semistable point of $(\bP^1)^6$;
this yields a well-defined point on $\fS$.  If $\beta(s)=q_1$
then we can identify $E_1=\bP(T_{q_1}\bP(V))=\bP(V/q_1)$ and
the images of the $q_j,j=2,\ldots,6$ in $\bP(V/q_1)$
with the intersections of the proper transforms of the $\fl_{1j}$
with $E_1$.  The rule
$$j(s)=(s,\fl_{12}\cap E_1,\ldots,\fl_{16}\cap E_1)$$
extends the definition of $j$ over $E_1 \subset S$.  
(This argument is very similar to the proof of Lemma~\ref{lemm:glue}.)  

An straightforward degree computation shows that
$j$ maps $S$ to a hyperplane section of $\fS$.  

For the final statement, the fifteen planes of $\fS$ cut out 
fifteen ordered lines of $S$.  The remaining lines form a double-six.  
\end{proof}

Let $Y'$ be a cubic threefold with ordinary double points at
$p_1,\ldots,p_6$; $\varpi$ induces a 
rational map $Y' \dashrightarrow S$ contracting the 
twisted quartic curves in $Y'$ to points of 
a smooth hyperplane section $\iota:S\hookrightarrow \fS$
with a distinguished double-six.  
After ordering the two sextuples of disjoint lines,
Proposition~\ref{prop:transversal} yields a determinantal cubic
hypersurface $Y$ with ordinary double points at $p_1,\ldots,p_6$
corresponding to the marked cubic surface $S$.  By
Proposition~\ref{prop:twisted}, the image of $Y$ under $\varpi$ is
a hyperplane section $\iota_2:S\hookrightarrow \fS$ with the
planes of $\fS$ tracing out the corresponding $15$ lines of $S$.  
Proposition~\ref{prop:tosegre} implies $\iota_1(S)=\iota_2(S)$
and thus $Y\simeq Y'$;  this yields an inverse to the morphism
(\ref{eqn:cubictosegre}).

\begin{rema}
The natural map 
$$
\left\{ \begin{array}{c}
\text{determinantal cubic threefolds} \\
\text{with ordinary double points } \\
\text{at $p_1,\ldots,p_6$} 
\end{array} \right\} \ra
\left\{ \begin{array}{c}
\text{cubic threefolds with} \\
\text{ordinary double points } \\
\text{at $p_1,\ldots,p_6$} 
\end{array} \right\}
$$
is not an isomorphism.  Under our identifications, these 
correspond to
$$
\left\{ \begin{array}{c}
\text{determinantal cubic surfaces} \\
\text{with a sextuple of ordered} \\
\text{lines}
\end{array} \right\} \ra
\left\{ \begin{array}{c}
\text{cubic surfaces with a} \\
\text{double-six of ordered}\\
\text{lines}
\end{array} \right\}
$$
which has degree two.  Indeed, this reflects the involution
(\ref{eqn:inv}) interchanging the sextuples of our double-six.  
\end{rema}

\section{Constructing flops}
\label{sect:example}
Let $X$ be a smooth cubic fourfold with hyperplane class $h$.  
\begin{prop} \label{prop:fourfold}
Assume that $X$ admits a hyperplane section $Y\subset X$
with six ordinary double points in linear general position.
Then $X$ contains two families of cubic scrolls $T$ and $T^{\vee}$, whose cycle classes satisfy
$$[T]+[T^{\vee}]=2h^2.$$
Conversely, if $X$ is a smooth cubic fourfold containing a smooth
cubic scroll $T$ then the 
hyperplane section
$$Y=X \cap \mathrm{span}(T)$$
has at least six double points, counted with multiplicities.  
\end{prop}
\begin{proof}
Assume that $X$ admits a hyperplane section $Y$ as above.  We may assume that
$Y$ is determinantal by 
Proposition~\ref{prop:Segre}.  Proposition~\ref{prop:scroll} guarantees that
$X$ contains two families of cubic scrolls, each parametrized by $\bP^2$.  
Given $T$ and $T^{\vee}$ from different families, we have
$$T \cup T^{\vee}=Y \cap Q$$
for some quadric hypersurface in $\bP^3$.  The equation on cycle classes
follows.

Now suppose that $X$ contains a smooth cubic scroll $T$ spanning the hyperplane
section $Y$.  In suitable coordinates, 
$$T=\{[x_0,\ldots,x_4]: \text{rank } 
\left( \begin{array}{ccc} x_0 & x_1 & x_2   \\
						   x_2 & x_3 & x_4 \end{array}
			\right) = 1 \}$$
and thus there exist linear forms $y_0,y_1,y_2$ in $x_0,\ldots,x_4$ such that
$$Y=\{[x_0,\ldots,x_4]: \mathrm{det }
\left( \begin{array}{ccc} x_0 & x_1 & x_2   \\
			   x_2 & x_3 & x_4 \\
			y_0 & y_1 & y_2  \end{array} \right)=0 \},$$
i.e., $Y$ is determinantal.  The double points correspond to the 
matrices of rank one (cf. Proposition~\ref{prop:linalg}).  
\end{proof}

\begin{theo} \label{theo:cubicscroll}
Let $X$ be a smooth cubic fourfold not containing a plane,
and $F$ its variety of lines.
Assume that $X$ admits a hyperplane section $Y$ with six ordinary
double points in linear general position.  Write
$$F(Y)=P \cup S' \cup P^{\vee}$$
with normalization $\bP^2 \sqcup S \sqcup \bP^2.$
Then there exist birational involutions
$$\iota,\iota^{\vee}:F \dashrightarrow F$$
which are regular away from $P\cup S'$ and $P^{\vee}\cup S'$ respectively.

Precisely, $\iota$ is factored as follows:
\begin{enumerate}
\item{Flop $P$ to get a new holomorphic symplectic fourfold
$F_1$;  the proper transform $S_1$ of $S$ is a plane in $F_1$.}
\item{Flop the $S_1$ in $F_1$ to get $F_2$, which is isomorphic
to $F$.}
\end{enumerate}
\end{theo}

\begin{proof}  
We construct $\iota$:  Let $[m] \in F$ be a line not contained in $F(Y)$.
Then $m \cap Y = \{y \}$, a nonsingular point of $Y$.  By Corollary~\ref{coro:fibration},
there exists a unique line $\ell^{\vee} \in P^{\vee}$ containing $y$.  Let $\Pi$ denote
the plane spanned by $\ell^{\vee}$ and $m$;  by assumption, $\Pi \not \subset X$.
Thus we have
\begin{equation} \label{eqn:barL}
\Pi \cap X= m \cup \ell^{\vee} \cup  \bar{m}
\end{equation}
for some line $\bar{m} \in F$.  Setting $\iota(m)=\bar{m}$, we get a morphism
$$\iota:F\setminus F(Y) \ra F.$$
Since (\ref{eqn:barL}) is symmetric in $m$ and $\bar{m}$, $\iota$ is
an involution.

As constructed, $\iota$ is not well-defined along $F(Y)$.  It remains to
show that it extends to
$m \in P^{\vee} \setminus (P\cup S)$.  Proposition~\ref{prop:components}
implies that $m$ does not contain any singularities of $Y$.  
The normal bundle to $m$ in $X$ is one of the following \cite[Proposition 6.19]{CG}
$$\cN_{m/X} \simeq \cO^{\oplus 2} \oplus \cO(1), \cO(-1) \oplus \cO(1)^{\oplus 2}.$$
Since $m$ does not contain any of the singularities of $Y$, 
Corollary ~\ref{coro:fibration} implies that we have the first case.
But then there exists a distinguished plane $\Pi$ with
$$\Pi \cap X= 2m \cup \bar{m},$$
i.e., $\Pi$ corresponds to the directions associated with the $\cO(1)$-summand.
Consider the correspondence
$$Z=\{(m,\ell^{\vee},\Pi): \Pi \cap X \supset m\cap \ell^{\vee}     \} 
\subset F \times P^{\vee}\times \bG(2,5);$$
the normal bundle computation guarantees that the projection
$$Z\ra F$$
is an isomorphism along $P^{\vee}\setminus (P\cup S)$.  
By definition, $\iota$ is regular on $Z$  
and thus at the generic point of $P^{\vee}$.  

We will use the following notation for our factorization
$$\begin{array}{rcccccccl}
     &         & F_{01} &          &    &         & F_{12} &         &              \\ 
     &\stackrel{\beta_{10}}{\swarrow} &        &\stackrel{\beta_{01}}{\searrow}  &
     &\stackrel{\beta_{21}}{\swarrow} &        &\stackrel{\beta_{12}}{\searrow}  & \\
F=F_0&         &        &          &F_1 &         &        &         & F_2 \\
     &\stackrel{\gamma_{10}}{\searrow} &        &\stackrel{\gamma_{01}}{\swarrow}  &
     &\stackrel{\gamma_{21}}{\searrow} &        &\stackrel{\gamma_{12}}{\swarrow}  & \\
     &         & \bar{F}_{01} &          &    &         & \bar{F}_{12} &         &             
\end{array}
$$
where $\beta_{10}$ blows up $P$, $\beta_{01}$
blows down the exceptional divisor of $\beta_{10}$, $\beta_{21}$
blows up $S_1$ (the proper transform of $S$), and $\beta_{12}$
blows down the exceptional divisor of $\beta_{12}$.  Here 
$\bar{F}_{01}$ and $\bar{F}_{12}$ denote the singular varieties
obtained by contracting $P$ and $S_1$ to a point.  In other words,
$F_2$ is obtained from $F_0$ by two Mukai flops.  Moreover, we will show that
$\iota:F \dashrightarrow F$ is resolved on passage to $F_2$, so the induced
$\iota_2:F \ra F$ is necessarily an isomorphism.

Let $P_1\subset F_1$ denote the plane that results from flopping $P$. 
\begin{lemm}
$S_1$ is isomorphic to $\bP^2$ and meets $P_1$ transversely at six points.  
\end{lemm}
\begin{proof}
Proposition~\ref{prop:components} describes how $S'$ and $P$ intersect:  $S'$
has two smooth branches meeting transversely in $F$, each of which meets
$P\simeq \bP^2$ in a line.  If $S^{\circ}$ is the smooth locus of $S'$ then
$$P\cap S^{\circ} \subset S^{\circ}$$
is Cartier, hence $\beta_{10}^{-1}(S^{\circ})\simeq S^{\circ}$.  However, $S'$
fails to be Cohen-Macaulay at the points of $S'\setminus S^{\circ}$, so any
Cartier divisor through these points would necessarily have an embedded point.
In particular,
$$P\cap S' \subset S'$$
is not Cartier at singular points of $S'$
and $\beta_{10}$ necessarily modifies $S'$ at these points.  

We claim that
the proper transform $S_{01}$ of $S'$ in $F_{01}$ is just $S$.
The easiest way to see this is through a local computation.  At each singularity of
$S'$ choose local coordinates $\{x_1,x_2,x_3,x_4\}$ such that
$$S'=\{x_1=x_2=0 \} \cup \{x_3=x_4=0 \} \quad P=\{x_2=x_3=0 \}.$$
The blow-up of $P$ has homogeneous equation 
$$Ax_3=Bx_2$$
and thus the proper transforms of the components of $S'$ are disjoint and
mapped isomorphically onto their images.

We next show that $\beta_{01}$ contracts the double-six in $S$ corresponding
to the intersection of $P$ with $S'$.  The key ingredient is the local description
of proper-transforms of Lagrangian submanifolds under Mukai flops given in \cite[\S 4.2]{Leung}:
Locally, a holomorphic-symplectic fourfold containing a plane looks like the 
total space of the cotangent-bundle of $\bP^2$.  In the cotangent bundle, a complex Lagrangian 
submanifold is modelled locally as the conormal sheaf $\cN^*_V$ of a complex submanifold $V\subset \bP^2$.
The Mukai flop is realized as the cotangent bundle of the dual plane $\cbP^2$;  
the proper transform of the Lagrangian submanifold looks locally like the conormal sheaf
$\cN^*_{\check{V}}$ of the projective dual $\check{V} \subset \cbP^2$.  Since each
branch of $S'$ looks locally like $\cN^*_{\ell}$ for a line $\ell \subset \bP^2$, 
its proper transform looks locally like $\cN^*_{[\ell]}$, where $[\ell] \in \cbP^2$
classifies $\ell$.  In particular, the proper transform $S_1$ meets $P_1$
in six points and $\beta_{01}:S_{01} \ra S_1$ blows down the double-six (corresponding to 
$S' \cap P$) to these points.  

The analysis in Proposition~\ref{prop:components} implies $S_1 \simeq \bP^2$, and thus
is a Lagrangian plane in $F_1$.  
\end{proof}

We define $F_2$ as the Mukai flop of this plane and let $S_2$
denote the resulting plane,
$P_2$ the proper transform of $P_1$, and
$$\phi:F_2 \dashrightarrow F$$
the composition of Mukai flops.  Now $P_2$ is isomorphic to a cubic surface,
meeting $S_2$ along a double-six as described before.

\

It remains to show that $\iota$
is resolved on $F_2$.  Let $g$ be the polarization on $F$ and the induced
divisor class;  it is globally generated except along $S \cup P$.  We shall
use the following result of Boucksom \cite{Boucksom}:  Let $Z$ be an irreducible holomorphic
symplectic variety and $f$ a divisor class on $Z$;  then $f$ is ample if and only
if $f$ meets each rational curve on $Z$ positively.  In particular, $\iota^*g$ 
necessarily meets some rational curves supported in $S\cup P$ negatively.  Consider the
map on one-cycles modulo homological equivalence induced by normalization and inclusion
$$j_*:N_1(\tilde{S},\bZ) \oplus N_1(P,\bZ) \ra N_1(F,\bZ).$$
This has a nontrivial kernel:  The description in Proposition~\ref{prop:components}
implies that 
$$j_*(E_i,0)=j_*(0,[\text{line}]), \quad i=1,\ldots,6$$
hence $j_*$ has rank two, and the image of the effective curve classes is spanned by 
$\lambda_1=j_*(0,[\text{line}])$ and $\lambda_2=j_*(\beta_{01}^*[\text{line}],0)$.
In particular, all the effective curve classes where $\iota^*g$ is negative
can be expressed as linear combinations of $\lambda_1$ and $\lambda_2$ with nonnegative
coefficients.  

We now analyze the pull-back $f=\phi^*\iota^*g$.  Since $\phi$ is an isomorphism away
from $P_2 \cup S_2$, any curve of $F_2$ not contained in $P_2 \cup S_2$ meets $f$
positively.  Just as before, the image of the effective curve classes in $P_2 \cup S_2$
are a cone generated by two classes $\lambda_1'$ and $\lambda_2'$.  However, using the
identifications
$$H_2(F_2,\bQ) \stackrel{\sim}{\lra} H^2(F_2,\bQ) \stackrel{\iota^*}{\lra} H^2(F,\bQ) 
\stackrel{\sim}{\lra} H_2(F,\bQ)$$
we see that $\lambda_1'=-\lambda_1$ and $\lambda_2'=-\lambda_2$.  It follows that $f$
is positive along {\em all} curves classes in $F_2$ and thus is ample.  

To reiterate, the composition
$$\phi \circ \iota:F \dashrightarrow F_2$$
is a birational map of smooth projective varieties, and takes the ample divisor
$g$ to an ample divisor $f$.  It follows that $\phi$ is an isomorphism.
\end{proof}

\begin{rema}
The strategy of our argument is due to Burns-Hu-Luo \cite{BHL}, who prove that 
any birational morphism of irreducible holomorphic symplectic varieties with
{\em normal} exceptional loci is a sequence of Mukai flops.  The normality
assumption can be eliminated (cf. \cite[1.2]{WW}).  
\end{rema}

\begin{prop} \label{prop:fixeddivisor}
Retain the notation of Theorem~\ref{theo:cubicscroll}.
Let $[v] \in \bP(V)$ and 
$T^{\vee}:=T_{v} \subset Y$ denote one of the cubic scrolls described
in Proposition~\ref{prop:scroll}.  The divisor
$$\tau^{\vee}=\{[m] \in F: m \cap T^{\vee} \neq \emptyset \} \subset F$$
is invariant under $\iota$.
\end{prop}
\begin{proof}
By Proposition~\ref{prop:scroll}, the ruling of $T^{\vee}$ is the
rational curve
$$\lambda_1^{\vee}=\{[v^{\vee}]:v^{\vee}(v)=0 \}
 \subset P^{\vee} \subset F(Y) \subset F.$$
If $m$ is incident to $T^{\vee}$ then $m$ meets some ruling 
$\ell_{[v^{\vee}]} \subset T^{\vee}$ where $[v^{\vee}] \in \lambda_1^{\vee}$.
Thus $\ell_{[v^{\vee}]}$ coincides with the line $\ell^{\vee}$
used in the construction of $\iota$.  Since the lines
$\{m,\ell_{[v^{\vee}]},\iota(m) \}$ are coplanar, we have
$\iota(m) \cap \ell_{[v^{\vee}]} \neq \emptyset$ and 
$\iota(m) \in \tau^{\vee}$.  
\end{proof}

\section{Cones of moving and ample divisors}
\label{sect:appl-cones}
Let $X$ be a smooth cubic fourfold and $F$ its variety of lines.  
The incidence correspondence
$$\begin{array}{ccccc}
  &   & Z & &  \\
  & \stackrel{\pi}{\swarrow} & & \stackrel{\psi}{\searrow} & \\
  X                    &   & &                       & F
\end{array}
$$
induces the {\em Abel Jacobi map} of integral Hodge structures \cite{BeauDonagi}
$$\alpha=\psi_*\phi^*:H^4(X) \ra H^2(F).$$
This is compatible with quadratic forms:  Writing $g=\alpha(h^2)$ we have
$$\left(g,g \right)=2\left<h^2,h^2\right>=6$$
and
$$\left(\alpha(z_1),\alpha(z_2)\right)=-\left<z_1,z_2\right> \text{ for all }
z_1,z_2 \in (h^2)^{\perp}.
$$
For general cubic fourfolds we have
$$H^4(X,\bZ) \cap H^{2,2}(X,\bC) = \bZ h^2,$$
but special cubic fourfolds admit additional algebraic cycles
\cite{Has00}:
\begin{prop} 
Let $\cC$ denote the moduli space of smooth cubic fourfolds
and $\cC_d \subset \cC$ the cubic fourfolds $X$ 
admitting a rank-two saturated lattice of
Hodge cycles
$$h^2 \in K_d  \subset H^4(X,\bZ) \cap H^{2,2}(X,\bC)$$
of discriminant $d$.  Then $\cC_d$ is nonempty if and only if
$$d \equiv 0,2\pmod{6}, \quad d>6.$$ 
In this case, $\cC_d$ is an irreducible divisor in $\cC$.  
\end{prop}

Assume that $X$ contains a smooth cubic scroll $T$ and
let $Y$ denote the hyperplane section of $X$ containing $T$.  
The intersection form $\left<,\right>$ on the middle cohomology of $X$
restricts to
$$K_{12}:=\begin{array}{c|cc}
  & h^2 & T \\
\hline
h^2 & 3 & 3 \\
T & 3 & 7 
\end{array}
$$
a lattice of discriminant $12$.  
Moreover, $\cC_{12}$ is the closure of the locus of cubic
fourfolds containing a cubic scroll or a hyperplane section
with six double points in general position.

Let $\tau=\alpha(T)$ so that the Beauville-Bogomolov form restricts to
$$J_{12}:=\begin{array}{r|cc}
  & g & \tau \\
\hline
g    & 6 & 6 \\
\tau & 6 & 2 
\end{array}
$$
a lattice of discriminant $-24$.  
We summarize elementary properties of this lattice:
\begin{enumerate}
\item{The elements of $J_{12} \otimes_{\bZ}\bR$ with nonnegative Beauville
form are a union of convex cones $\cP \cup -\cP$ where
$$\cP=\Cone\left(g-(3-\sqrt{6}\tau,(3+\sqrt{6})\tau -g\right).$$  }
\item{$J_{12}$ does not represent $-2$ or $0$.}
\item{The automorphism group of $J_{12}$ 
is isomorphic to the direct product of $\left< \pm 1\right>$
and the infinite dihedral group
$$\Gamma:=\left<R_1,R_2:R_1^2=R_2^2=1\right>$$
where the reflections $R_1,R_2$ can be written
$$ \begin{array}{rclcrcl}
R_1(g)&=&g & & R_1(\tau)&=&2g-\tau \\
R_2(g)&=&-g+6\tau & &	R_2(\tau)&=&\tau.
\end{array}$$
$\Gamma$ consists of the automorphisms taking $\cP$ to itself.}
\item{$J_{12}$ represents $-10$.  We list $(-10)$-classes
with positive intersection with $g$:
$$
\begin{array}{rclcrcl}
\multicolumn{3}{c}{ \text{class} \quad \quad} & & \multicolumn{3}{c}{\text{intersection with }g} \\
\hline 
&\vdots& & & &\vdots & \\
\rho^{\vee}_3&=&16\tau-3g & & \left(\rho^{\vee}_3,g\right)&=&78 \\
\rho^{\vee}_2&=&4\tau-g & &\left( \rho^{\vee}_2,g\right)&=&18 \\
\rho^{\vee}_1&=&2\tau-g & &\left( \rho^{\vee}_1,g\right)&=&6 \\
\rho_1&=&3g-2\tau & &\left( \rho_1,g\right)&=&6 \\
\rho_2&=&7g-4\tau & &\left( \rho_2,g\right)&=&18 \\
\rho_3&=&29g-16\tau & & \left( \rho_3,g\right)&=&78 \\
&\vdots & & & & \vdots &
\end{array}$$
For $j\ge 3$ we define recursively 
$$
\rho_j=(R_1R_2)\rho_{j-2} \,\, \text{ and }\,\, \rho_j^{\vee}=(R_2R_1)\rho_{j-2}^{\vee}.
$$  
}
\item{The element $R_1R_2$ has infinite order and acts
on the $(-10)$-classes with orbits:
$$
\begin{array}{c}
\{\ldots, \rho^{\vee}_3,\rho^{\vee}_1,-\rho_2,\ldots \} \quad
\{\ldots, -\rho^{\vee}_3,-\rho^{\vee}_1,\rho_2,\ldots \} \\
 \{\ldots,-\rho_2^{\vee},\rho_1,\rho_3\ldots \} \quad
 \{\ldots,\rho_2^{\vee},-\rho_1,-\rho_3\ldots \}.
\end{array}
$$
The element $R_1$ has order two and acts via
$$R_1(\rho_i)=\rho_i^{\vee}$$
for each $i$.}  
\end{enumerate}

\begin{prop} \label{prop:nef1}
Assume that $X$ is a smooth cubic fourfold containing a smooth cubic
scroll $T$ with
$$H^4(X,\bZ) \cap H^{2,2}(X,\bC)=\bZ h^2 + \bZ T$$
or equivalently
$$H^2(F,\bZ) \cap H^{1,1}(F,\bC)=\bZ g + \bZ \tau.$$
The nef cone of $F$ equals  
the cone dual to $\Cone(\rho_1,\rho_1^{\vee})$, i.e., 
$\Cone(\alpha_1,\alpha_1^{\vee})$ where
$\alpha_1=7g-3\tau$ and $\alpha^{\vee}_1=g+3\tau$.
\end{prop}
\begin{proof}
The main theorem of \cite{HT07} asserts that a divisor class $f$ on $F$
is nef if $\left(f,\rho\right) \ge 0$ for each divisor
class $\rho$ on $F$ satisfying $\left(g,\rho \right)>0$ and either
$\left(\rho,\rho\right) \ge -2$ or
$\left(\rho,\rho\right)  -10$ and  $\left(\rho,H^2(F,\bZ)\right)=2\bZ$.  
Thus the nef cone contains $\Cone(\alpha_1,\alpha_1^{\vee})$.

It remains to show that $\alpha_1$ and $\alpha_1^{\vee}$ are
at the boundary of the nef cone.  We show they induce nontrivial
contractions of $F$.
It suffices to prove this for a generic cubic fourfold
containing a cubic scroll, so we may assume $Y$ has
exactly six ordinary double points in linear general position.  It follows
that there exist Lagrangian planes 
$$P,P^{\vee} \subset Y\subset F$$ 
such that lines
$\lambda_1\subset P$ and $\lambda^{\vee}_1 \subset P^{\vee}$ 
both have degree three.
The classes of these lines are dual to 
$\rho_1$ and $\rho_1^{\vee}$ respectively.
The divisors $\alpha_1$ and $\alpha_1^{\vee}$ induce small contractions 
$$\gamma_{10}:F=F_0\ra \bar{F}_{01}, \quad
\gamma^{\vee}_{01}:F_0 \ra \bar{F}_{10}^{\vee}$$
of $P$ and $P^{\vee}$ respectively.  
\end{proof}

We may flop the plane $P$ (resp. $P^{\vee}$) in $F$
to obtain new holomorphic symplectic fourfold $F_1$ (resp. $F_1^{\vee}$).
The birational maps between $F,F_1$, and $F_1^{\vee}$ induce
identifications of their Picard groups.  
The first step in analyzing their nef cones is
to enumerate the orbit of $\alpha_1$ under $\Gamma$:
$$
\begin{array}{lclcrcl}
\multicolumn{3}{c}{ \text{class} \quad \quad} & & \multicolumn{3}{c}{\text{intersection with }g} \\
\hline 
&\vdots& & & &\vdots & \\
\alpha^{\vee}_2=R_2R_1(\alpha_1)&=&9\tau-g & &\left( \alpha^{\vee}_2,g\right)&=&48 \\
\alpha^{\vee}_1=R_1(\alpha_1)&=&g+3\tau & &\left( \alpha^{\vee}_1,g\right)&=&24 \\
\alpha_1&=&7g-3\tau & &\left( \alpha_1,g\right)&=&24 \\
\alpha_2=R_1R_2(\alpha^{\vee}_1)&=&17g-9\tau & &\left( \alpha_2,g\right)&=&48 \\
&\vdots & & & & \vdots &
\end{array}$$
For $j \ge 3$ we define recursively
$$\alpha_j=R_1R_2(\alpha_{j-2})  \,\, \text{ and }\,\, 
\alpha^{\vee}_{j}=R_2R_1(\alpha^{\vee}_{j-2}).$$
Observe that $\left(\alpha_i^{\vee},\rho_i^{\vee}\right)=\left(\alpha_i,\rho_i\right)=0$ for each $i\ge 1$.  

\begin{prop} \label{prop:nef2}
The nef cone of $F_1$ (resp. $F^{\vee}_1$) equals $\Cone(\alpha_1,\alpha_2)$
(resp. $\Cone(\alpha_2^{\vee},\alpha_1^{\vee})$.)  
\end{prop}
\begin{figure}
\begin{center}
\includegraphics{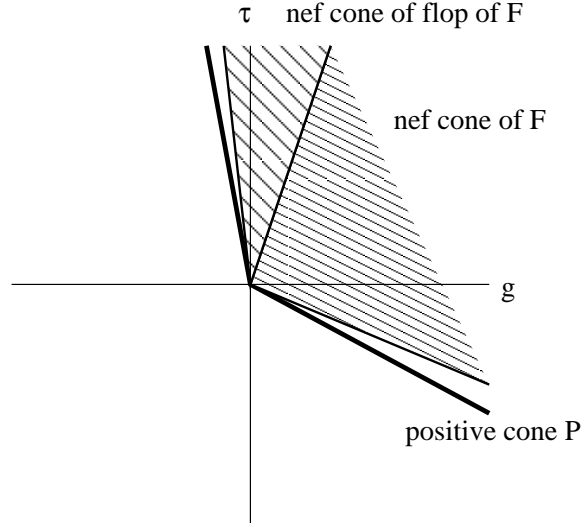}
\end{center}
\caption{The nef cones of $F$ and $F_1^{\vee}$}
\end{figure}
\begin{proof}
Since $F_1$ is the flop of $F$ along $P$, it is clear that $\alpha_1$
(which induces the contraction of $P$)
is one generators of the nef cone.  The factorization
given in Theorem~\ref{theo:cubicscroll}
shows that $\iota$ induces a {\em regular} involution
on $F_1$.  Proposition~\ref{prop:fixeddivisor} implies that $\iota$
fixes the divisor
$$\tau^{\vee}=\alpha(T^{\vee})=\alpha(2h^2-T)=2g-\tau,$$
where the middle equality uses Proposition~\ref{prop:fourfold}.  
Consequently, $\iota$ acts on $J_{12}$
via the reflection $R_3$ through the line orthogonal to $2g-\tau$
$$R_3(g) = 11g-6\tau, \quad R_3(\tau)=20g-11\tau.$$
The second generator of the nef cone of $F_1$ is thus
$$R_3(\alpha_1)=17g-9\tau=\alpha_2.$$
\end{proof}

Theorem~\ref{theo:cubicscroll} gives 
isomorphisms between every second model, so Proposition~\ref{prop:nef1}
and \ref{prop:nef2} suffice to describe the nef cone of
{\em every} model of $F$.    We summarize our whole discussion:
\begin{theo} \label{theo:main}
Suppose that $X$ is a smooth cubic fourfold containing a smooth cubic
scroll $T$ with
$$H^4(X,\bZ) \cap H^{2,2}(X,\bC)=\bZ h^2 + \bZ T$$
and let $F=F_0$ denote the variety of lines on $X$.  
Then we have an infinite sequence of Mukai flops
$$\cdots F^{\vee}_2  \dashrightarrow F^{\vee}_1 \dashrightarrow F_0 \dashrightarrow
F_1 \dashrightarrow F_2 \cdots$$
with isomorphisms between every other flop in this sequence
$$\cdots F^{\vee}_{2} \stackrel{\sim}{\ra} F_0 \stackrel{\sim}{\ra} F_2  \cdots \quad \text{ and } \quad
\cdots F^{\vee}_{1} \stackrel{\sim}{\ra} F_1  \cdots$$

The positive cone of $F$ can be expressed as the union of the nef cones
of the models $\{ \cdots, F^{\vee}_1,F_0,F_1, \cdots \}$:
$$\cdots, \Cone(\alpha^{\vee}_2,\alpha^{\vee}_1), \ \Cone(\alpha^{\vee}_1,\alpha_1),\ \Cone(\alpha_1,\alpha_2),
\cdots$$
The isomorphisms induce an action of
$\bZ\simeq \left<R_1R_2\right> \subset \Gamma$ on the Picard group
of $F$.  
\end{theo}
\begin{figure}[h]
\begin{center}
\includegraphics{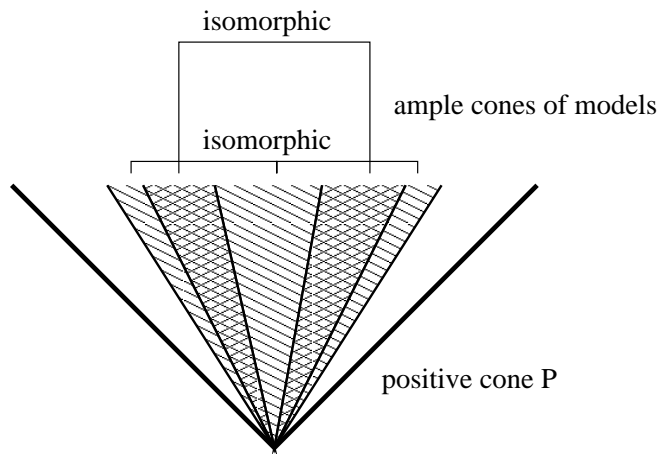}
\end{center}
\caption{Schematic illustration of the partition of the positive cone 
into ample cones for isomorphism classes of minimal models}
\end{figure}
This verifies the conjectures of \cite{HT01} for the birational models of $F$.  
It also illustrates the Finiteness of Models Conjecture of Kawamata and Morrison--here
we have two birational models up to isomorphism.

\begin{rema}[Application to rational points]
Let $F$ be the variety of lines of a cubic fourfold $X$ containing
a cubic scroll $T$, both defined over a field $k$.  Assume that the
hyperplane section containing $T$ has precisely six ordinary double
points in linear general position and $X$ does not contain a plane.
Then $k$-rational points on $F$ are Zariski dense.  Indeed, the 
infinite collection of Lagrangian planes defined over $k$ is 
Zariski dense.  

If the Picard group of $F$ has rank two then
\begin{itemize}
\item $F$ does not admit regular automorphisms, and
\item $F$ is not birational to an abelian fibration.
\end{itemize}

Potential density of rational points on varieties of lines on {\em generic}
cubic fourfolds over number fields has recently been established in \cite{AV}.  
\end{rema}

\bibliographystyle{plain}
\bibliography{flopping}

\end{document}